\documentclass[a4paper,11pt]{article}
\usepackage{graphicx}
\usepackage{amssymb}
\usepackage{amsmath}
\usepackage{amsfonts}
\usepackage{amsthm}
\usepackage{epsfig}
\usepackage[all]{xy}
\usepackage{mathdots}
\usepackage[english]{babel}

\newcommand{\mR}{\mathbb{R}}

\newcommand{\mcH}{\mathcal{H}}

\newcommand{\mcU}{\mathcal{U}}
\newcommand{\mcP}{\mathcal{P}}

\newcommand{\mN}{\mathbb{N}}
\newcommand{\mZ}{\mathbb{Z}}

\newcommand{\mE}{\mathbb{E}}

\newcommand{\mfg}{\mathfrak{g}}

\DeclareMathOperator{\Ker}{Ker}
\DeclareMathOperator{\Img}{Im}

\DeclareMathOperator{\osp}{\mathfrak{osp}}
\DeclareMathOperator{\gl}{\mathfrak{gl}}
\DeclareMathOperator{\msl}{\mathfrak{sl}}
\DeclareMathOperator{\so}{\mathfrak{so}}
\DeclareMathOperator{\spp}{\mathfrak{sp}}

\theoremstyle{plain}
\newtheorem{thm}{Theorem}
\newtheorem*{thmA}{Theorem A}
\newtheorem*{thmB}{Theorem B}

\newtheorem{cor}{Corollary}
\newtheorem{lemma}{Lemma}
\newtheorem{prop}{Proposition}

\theoremstyle{definition}

\newtheorem{remark}{Remark}

\newcommand{\pa}{\partial}

\title{Fischer decomposition for polynomials on superspace}
\author{Roman L\'{a}vi\v{c}ka$^\dagger$\footnote{email: {\tt lavicka@karlin.mff.cuni.cz}}
\and Dalibor \v{S}m\'{\i}d$^\dagger$}

\date{\small{
$^\dagger$ Charles University in Prague, Faculty of Mathematics and Physics,\\
Sokolovsk\'a 83, 186 75 Praha, Czech Republic}}

\begin{document}

\maketitle

{\small
\abstract{Recently, the Fischer decomposition for polynomials on superspace $\mR^{m|2n}$ (that is, polynomials in $m$ commuting and $2n$ anti-commuting variables) has been obtained unless the superdimension $M=m-2n$ is even and non-positive.  In this case, it turns out that the Fischer decomposition of polynomials into spherical harmonics is quite analogous as in $\mR^m$ and it is an irreducible decomposition under the natural action of Lie superalgebra $\osp(m|2n)$. In this paper, we describe explicitly the Fischer decomposition in the exceptional case when $M\in -2\mN_0$. In particular, we show that, under the action of $\osp(m|2n)$, the Fischer decomposition is not, in general, a decomposition into irreducible but just indecomposable pieces.}

}

\section{Introduction}

Recently, harmonic analysis has been developed on superspace $\mR^{m|2n}$ generated by  $m$ commuting (bosonic) and $2n$ anti-commuting (fermionic) variables, see \cite{Cou, Cou_JPAA, CDS, CDS_London, DES, DS3, Luo,LX, Zhang}. The Laplace operator $\Delta$, the square of norm $R^2$ and the Euler operator $\mE$ were generalized from the Euclidean space $\mR^m$ to the superspace $\mR^{m|2n}$ so that they still generate the Lie algebra $\msl(2)$. In particular, we have the commutation relation $[\Delta, R^2]= 4 \mE + 2 M$ 
where $M=m-2n$ is the so-called superdimension of $\mR^{m|2n}$. It was shown that $\Delta$, $R^2$ and $\mE$ are all invariant operators on the space $\mcP$ of scalar valued polynomials on $\mR^{m|2n}$ with respect to the natural action of the Lie superalgebra $\osp(m|2n)$. See the next section for more details.

Unless $M\in-2\mN_0$, the Fischer decomposition of $k$-homogeneous polynomials $\mcP_k$ into spherical harmonics (that is, separation of variables) has the usual form
\begin{equation}\label{I_Fischer0}
\mcP_k = \mcH_k \oplus R^2 \mcP_{k-2}\text{\ \ and\ \ }\mcP_{k}=\bigoplus_{j=0}^{\lfloor k/2\rfloor} R^{2j}\mcH_{k-2j}
\end{equation}
(see \cite[Theorem 3]{DS3} and \cite{Luo}).
Here $\mcH_k=\Ker(\Delta)\cap \mcP_k$ stands for the space of $k$-homogeneous spherical harmonics. Furthermore, the Fischer decomposition is an irreducible decomposition of $\mcP_k$ under the action of $\osp(m|2n)$.
Even, it turns out that $\osp(m|2n)$ and $\msl(2)$ form the Howe dual pair for polynomials on $\mR^{m|2n}$ unless $M\in-2\mN_0$, see \cite{Cou}. 
For an account of the Howe duality, we refer to  \cite{CW, Howe1, Howe2, LS, Serg1, Serg2}.

The exceptional case when $M\in-2\mN_0$ is more interesting. In this case, it is easy to observe that the Fischer decomposition can not have the form \eqref{I_Fischer0}. Indeed, it follows from the relation $[\Delta, R^2]= 4 \mE + 2 M$ that
$$\mcH^0_k=R^{2k+M-2}\mcH_{2-M-k}\subset\mcH_k$$
whenever $k\in I_M=\{k\in\mN_0|\ 2-M/2\leq k\leq 2-M\}$. It was proved that the $\osp(m|2n)$-module $\mcH_k$ is always irreducible except for the exceptional indices $k\in I_M$ when $\mcH_k$ is still indecomposable and the module $\mcH_k^0$ is irreducible and is the maximal $\osp(m|2n)$-submodule of $\mcH_k$ (see \cite{Cou} or Theorem B below).

In this paper, we describe explicitly the Fischer decomposition including the exceptional case $M\in -2\mN_0$. Indeed,  
by Theorem \ref{Fischer}, we have in general that
\begin{equation}\label{I_Fischer}
\mcP_k = \tilde\mcH_k \oplus R^2\Delta R^2 \mcP_{k-2}\text{\ \ and\ \ }\mcP_{k}=\bigoplus_{j=0}^{\lfloor k/2\rfloor} (R^2\Delta R^2)^{j}\tilde\mcH_{k-2j}
\end{equation}
where  $\tilde\mcH_k=\Ker(\Delta R^2\Delta)\cap \mcP_k$. Moreover, in Theorem \ref{tildeH}, we describe the structure of $\osp(m|2n)$-module $\tilde\mcH_k$. In particular, we show that $\tilde\mcH_k=\mcH_k$ except for the exceptional indices $k\in I_M$ when $\tilde\mcH_k$ is only indecomposable and the module $\mcH_k$ is the maximal $\osp(m|2n)$-submodule of $\tilde\mcH_k$.

In \cite{CD}, spinor valued polynomials on superspace $\mR^{m|2n}$ are studied. The super Dirac operator is introduced and the Fischer decomposition is described again only when $M\not\in-2\mN_0$. In a~forthcoming paper, we are going to investigate the exceptional case 
$M\in-2\mN_0$. The further difficulty lies in the fact that super spinor spaces are infinite dimensional. 

The paper is organized as follows. In Section 2, we introduce some notation and recall  known results from harmonic analysis on superspace.
In Section 3, we state our main results, namely, Theorems \ref{Fischer} and \ref{tildeH}. In Section 4, we give their proofs.

\section{Harmonic analysis on superspace}


In this section, we recall  known results from harmonic analysis on superspace. 

For an account of Lie superalgebras, we refer to \cite{CW}.
Let $V=V_0\oplus V_1$ be a~finite dimensional superspace (that is, $\mZ_2$-graded vector space) and let $\gl(V)$ stand for the Lie superalgebra of endomorphisms of $V$
endowed with the supercommutator 
$$[A,B]=AB-(-1)^{|A||B|}BA$$ 
where $|A|$ is the parity of a~homogeneous operator $A$. Moreover, assume that we have a~scalar superproduct $g$ on $V$, that is, 
$g$ is a non-degenerate bilinear form on $V$ such that $g|_{V_0}$ is symmetric, $g|_{V_1}$ is skew symmetric
and,  for each $v_0\in V_0$ and $v_1\in V_1$, $g(v_0,v_1)=0=g(v_1,v_0)$.  
Then the Lie superalgebra $\osp(V,g)$ consists of all operators $A\in\gl(V)$ which preserve $g$, that is, for $x,y\in V$,
$$g(Ax,y)+(-1)^{|A||x|}g(x,Ay)=0.$$ 
In this paper, we deal with the supersymmetric tensor product 
$$S(V)=\bigoplus_{k=0}^{\infty} S(V)_k\text{\ \ with\ } S(V)_k=\odot^k V$$
of the basic representation $V$ for $\mfg=\osp(V,g)$. Recall that the action of $\mfg$ on the tensor product $U\otimes W$ of $\mfg$-representations $U,W$ is defined by
$$A(u\otimes w)=(Au)\otimes w+(-1)^{|A||u|}u\otimes (Aw)$$
for $A\in\mfg$  and $u\in U$, $w\in W$. The symmetric tensor product $u\odot w$ is given by
$$u\odot w=(u\otimes w+(-1)^{|u||w|}w\otimes u)/2$$
for homogeneous $u\in U$, $w\in W$. 

To be more explicit, let the superspace $V$ be $\mR^{m|2n}$, that is, $V_0=\mR^{m}$
and $V_1=\mR^{2n}$. With respect to the standard basis, we denote the coordinates of $x\in\mR^{m|2n}$ as
$$x=(X_1,\ldots,X_{m+2n})=(x_1,\ldots,x_m,\theta_1,\ldots,\theta_{2n})$$
and we assume that the matrix of the given metric $g$ is the block diagonal matrix
$$g=(g^{ij})=
\begin{pmatrix}
E_m&\ \\
\ &J_{2n}\\
\end{pmatrix}
$$
where $E_m$ is the identity matrix of size $m$ and $J_{2n}$ is the square matrix of size $2n$ given by
$$
J_{2n}=\frac{1}{2}
\begin{pmatrix}
0&-1&&&& \\
1&0&&&\\
&&&\ddots&&\\
&&&&0&-1\\
&&&&1&0
\end{pmatrix}
$$
We write $\osp(m|2n)$ for $\osp(V,g)$.
Then the supersymmetric tensor product $S(V)$ can be realized as the space of polynomials on $\mR^{m|2n}$
$$\mcP=\mR[x_1,\ldots,x_m]\otimes\Lambda_{2n}.$$
Here $\mR[x_1,\ldots,x_m]$ are polynomials in $m$ commuting (bosonic) variables $x_1,\ldots,x_m$ and $\Lambda_{2n}$ is the Grassmann algebra generated by $2n$ anticommuting (fermionic) variables $\theta_1,\ldots,\theta_{2n}$. Moreover, the partial derivative $\pa_{X_j}$ and (the left multiplication by) $X_j$ are both homogeneous operators on $\mcP$ with the parity $|\pa_{X_j}|=|X_j|=[j]$ where $[j]=0$ for $j=1,\ldots,m$ and $[j]=1$ for $j=m+1,\ldots,m+2n$.
We have that 
$$[\pa_{X_i},X_j]=\delta_{ij}.$$
Using the metric $g=(g^{ij})$ one can raise indices, in particular,
$${X^j}=\sum_{i=1}^{m+2n}{X_i}\;g^{ij}.$$
Then, obviously, we have that
$$(X^1,\ldots,X^{m+2n})=(x_1,\ldots,x_m,\theta_2/2,-\theta_1/2,\ldots,\theta_{2n}/2,-\theta_{2n-1}/2),$$
$$(\pa_{X^1},\ldots,\pa_{X^{m+2n}})=(\pa_{x_1},\ldots,\pa_{x_m},2\pa_{\theta_2},-2\pa_{\theta_1},\ldots,2\pa_{\theta_{2n}},-2\pa_{\theta_{2n-1}}).$$
The action of $\osp(m|2n)$ on $\mcP$ is generated by
$$L_{ij}=X_i\pa_{X^j}-(-1)^{[i][j]}X_j\pa_{X^i}$$
for $1\leq i\leq j\leq m+2n$ (see \cite{CDS_invariant,Zhang}). 

As in the classical case, we have the Lie algebra $\msl(2)$ formed by $\osp(m|2n)$-invariant operators on $\mcP$, namely, the super Laplace operator, the super Euler operator and the square of super norm:
$$\Delta=\sum_{j=1}^{m+2n}\pa_{X^j}\pa_{X_j},\ \ \mE=\sum_{j=1}^{m+2n}X_j\pa_{X_j},\ \ R^2=\sum_{j=1}^{m+2n}{X^j}{X_j}.$$
Indeed, we have 
$[\Delta/2, R^2/2]=\mE+M/2$ and $$[\Delta/2, \mE+M/2]=2\Delta/2,\ \ \  [R^2/2, \mE+M/2]=-2R^2/2$$
where $M=m-2n$ is the so-called superdimension of $\mR^{m|2n}$. 
Moreover,
each of these operators has the bosonic and the fermionic part
$$\Delta=\Delta_b+\Delta_f\text{\ \ with\ \ }
\Delta_b=\sum_{j=1}^{m} \pa_{x_j}^2,\ \Delta_f=-4\sum_{j=1}^{n}\pa_{\theta_{2j-1}}\pa_{\theta_{2j}};$$
$$\mE=\mE_b+\mE_f\text{\ \ with\ \ }
\mE_b=\sum_{j=1}^{m} x_j\pa_{x_j},\ \mE_f=\sum_{j=1}^{2n}\theta_j\pa_{\theta_j};$$
$$R^2=r^2+\theta^2\text{\ \ with\ \ }
r^2=\sum_{j=1}^{m} {x^2_j},\ \theta^2=-\sum_{j=1}^{n}{\theta_{2j-1}}{\theta_{2j}}.$$
Let us remark that, by \cite{JG}, the Casimir operator of degree 2 for $\osp(m|2n)$ is given by
\begin{equation}\label{Casimir}
C=R^2\Delta-\mE(M-2+\mE).
\end{equation}
The space of $k$-homogeneous polynomials is given by
$$\mcP_k=\{P\in\mcP|\ \mE P=kP\}$$ and, for an operator $A$ on $\mcP$, we define
$$\Ker A=\{H\in\mcP|\ A H=0\}\text{\ \ and\ \ }\Ker_k A=\Ker(A)\cap\mcP_k.$$ 
The important role is played by the space of $k$-homogeneous spherical harmonics
$$\mcH_k=\Ker_k \Delta.$$ In the purely bosonic and fermionic case when $n=0$ and $m=0$, we write $\mcH_k^b$ and $\mcH_k^f$ for $\mcH_k$, respectively.

In \cite[Theorem 3]{DS3} and \cite{Luo}, the Fischer decomposition of polynomials on $\mR^{m|2n}$ into spherical harmonics is obtained unless  
$M\in-2\mN_0$. Indeed, we have

\begin{thmA} Let $M\not\in-2\mN_0$. Then, for each $k\in\mN_0$, we have that
$$\mcP_k = \mcH_k \oplus R^2 \mcP_{k-2}.$$
\end{thmA}

\begin{remark}
By induction, it follows from Theorem A that, for each $k\in\mN_0$,
$$\mcP_{k}=\bigoplus_{j=0}^{\lfloor k/2\rfloor} R^{2j}\mcH_{k-2j}.$$ 
Hence the Fischer decomposition looks like the classical (purely bosonic) one. In Figure \ref{Fig_I}, all the summands of $\mcP_k$ are contained in the $k$-th column. Each row yields an infinite dimensional representation of $\msl(2)$ generated by the operators $\Delta$  and $R^2$.

Let us remark that, 
in the case  $m=1$, we have  $\mcH_k\not=0$  if and only if $k=0,\ldots,2n+1$ (see \cite{Cou}). So in this case,  for all $k>2n+1$, the $k$-th rows in Figure \ref{Fig_I} are missing. 
\end{remark}

\begin{figure}[h]

\caption{The Fischer decomposition for $M\not\in -2\mN_0$.}\label{Fig_I}

$$\xymatrix@!C=5pt@R=1pt{
\mcP_0 & \mcP_1 & \mcP_2 & \mcP_3 & \mcP_4 & \mcP_5 & \cdots \\
\\
\mcH_0 & & \ar[ll]_{\Delta}  \ar[rr]^{R^2} R^2\mcH_0 & & R^4\mcH_0 & & \cdots \\
& \mcH_1 & &  R^2\mcH_1 & & R^4\mcH_1  \\
& & \mcH_2 & & R^2\mcH_2 & & \cdots   \\
& & & \mcH_3 & &  R^2\mcH_3  \\
& & & & \mcH_4 & & \cdots \\
& & & & & \mcH_5 & & \\
& & & & & & \ddots & & \\
}$$

\end{figure}

The case when $m=0$ is well-known, see e.g.\ \cite[\S\;4.2]{DES}. This is the purely fermionic case and $M=-2n$.
For example, the Fischer decomposition is depicted in Figure \ref{Fig_II}. In this case, the diagram has $n+1$ rows and each row forms a~finite dimensional $\msl(2)$-representation.

\begin{figure}[h]

\caption{The Fischer decomposition for $M\in -2\mN_0$ and $m=0$.}\label{Fig_II}

$$\xymatrix@!C=5pt@R=1pt{
\mcP_0 & \mcP_1 & \mcP_2 & \mcP_3 & \cdots & \cdots & \mcP_{2n-2} \ \ \ & \mcP_{2n-1} & \mcP_{2n} \\
\\
\mcH_0 & & \ar[ll]_{\Delta}  R^2\mcH_0 & & \cdots & & R^{2(n-1)}\mcH_0 & & R^{2n}\mcH_0 \\
&  \ar[rr]^{R^2}  \mcH_1 & & R^2\mcH_1 & & \cdots & & R^{2(n-1)}\mcH_1 \\
& & \mcH_2 & & \cdots & & R^{2(n-2)}\mcH_2 \\
& & & \ddots & & \iddots \\
& & & & {\mcH}_n\\
}$$

\end{figure}

In what follows, we assume that $m\not=0$ unless otherwise stated. 
The structure of $\mcH_k$ is known when it is viewed as a~module under the action of the whole $\osp(m|2n)$ or only its even part $\so(m)\times\spp(2n)$.
For the following result, we refer to  \cite[Theorem 4]{DES}. 

\begin{prop}\label{decomp0} 
Under the action of $\so(m)\times\spp(2n)$, the module $\mcH_k$ has a~multiplicity free irreducible decomposition
$$\mcH_k\simeq\bigoplus_{j=0}^{\min(n,k)}\bigoplus_{\ell=0}^{\min(n-j,\lfloor\frac{k-j}{2}\rfloor )} \mcH^b_{k-2\ell-j}\otimes \mcH^f_j.$$
\end{prop}

Now we recall that $\osp(m|2n)$-module $\mcH_k$ is irreducible except for some 'exceptional' indices $k$.
For a~given superdimension $M$, we define the set of exceptional indices by $I_M=\emptyset$ if $M\not\in-2\mN_0$, and
\begin{equation}\label{IM}
I_M=\{k\in\mN_0|\ 2-M/2\leq k\leq 2-M\}\text{\ if\ } M\in-2\mN_0.
\end{equation}

Denote by $L^{m|2n}_{\lambda}$ an $\osp(m|2n)$-irreducible module with the highest weight $\lambda$. We use the simple root system of $\osp(m|2n)$ as in \cite{Zhang, Cou}. This is not the standard choice \cite{Kac} but is more convenient for our purposes.

\begin{thmB}{\rm (\cite{Cou,Luo})} Let the set $I_M$ be defined as in \eqref{IM} and denote $\mcH^0_k=R^2\Ker_{k-2} (\Delta R^2)$.

\smallskip\noindent
(i) Let $k\not\in I_M$. Then $\mcH_k\simeq L^{m|2n}_{(k,0\ldots,0)}$  and $\mcH^0_k=0$. 

\smallskip\noindent
(ii) Let $k\in I_M$. Then $\mcH^0_k=R^{2k+M-2}\mcH_{2-M-k}$ and $\mcH^0_k$
is a~proper subset of $\mcH_k$.
Moreover, under the action of $\osp(m|2n)$, the module
$\mcH_k$ is indecomposable and its composition series is
$$\mcH^0_k\simeq L^{m|2n}_{(2-M-k,0\ldots,0)},\ \ \mcH_k/\mcH^0_k\simeq L^{m|2n}_{(k,0\ldots,0)}.$$ 

\end{thmB}

\section{The main results}

In this section, we state our main results which generalize Theorems A and B. Their proofs are given in the next section.

First we describe the Fischer decomposition including the exceptional case $M\in-2\mN_0$.

\begin{thm}\label{Fischer} For each $k\in\mN_0$, we have that
$$\mcP_k = \tilde{\mcH}_k \oplus R^2 \Delta R^2 \mcP_{k-2}$$
where $\tilde{\mcH}_k = \Ker_k (\Delta R^2 \Delta)$.
\end{thm}

By induction, it follows from Theorem \ref{Fischer} that, for each $k\in\mN_0$,
\begin{equation}\label{eq_Fischer}
\mcP_{k}=\bigoplus_{j=0}^{\lfloor k/2\rfloor} (R^2 \Delta R^2)^j\tilde\mcH_{k-2j}.
\end{equation} 
Let us note that 
some direct summands might be trivial. Indeed, we have

\begin{cor}\label{cor_Fischer} Let $k\in\mN_0$ and let the exceptional indices $I_M$ be defined as in \eqref{IM}.
Denote $N_k=\{k-2j|\ j=0,\ldots,\lfloor k/2\rfloor\}$ and $\tilde J_k=N_k\cap I_M$. 
Then we have that
\begin{equation}\label{eq_Fischer+}
\mcP_{k}=\bigoplus_{\ell\in \tilde J_k} R^{k-\ell}\tilde\mcH_{\ell} \oplus
\bigoplus_{\ell\in J_k} R^{k-\ell}\mcH_{\ell} 
\end{equation}
where $J_k=N_k\setminus(\tilde J_k\cup J_k^0)$ with $J_k^0=\{2-M-\ell|\ \ell\in\tilde J_k\}$. 
\end{cor}

The particular case $M=-4$ is depicted in Figure \ref{Fig_III}. The exceptional indices are $I_{-4}=\{4,5,6\}$.
Notice that the first three rows look like the diagram for the purely fermionic Fischer decomposition  with $n=2$ (see Figure \ref{Fig_II}).
Other rows are infinite as in the classical case (see Figure \ref{Fig_I}). The $k$-th row starts with $\mcH_k$ except for the exceptional indices $k\in I_{-4}$ when it starts with ${\tilde{\mcH}_k}$. Moreover, the operators $\tilde\Delta=\Delta R^2\Delta$ and $\tilde R^2=R^2\Delta R^2$ act in a~given row like $\Delta$ and $R^2$  in the classical or purely fermionic case even if they do not generate $\msl(2)$.  
For example, we have $\tilde R^2 R^4\mcH_0=0$.

\begin{figure}[h]

\caption{The Fischer decomposition for $M=-4$ and $m\not=0$.}
\label{Fig_III}

$$\xymatrix@!C=5pt@R=1pt{
\mcP_0 & \mcP_1 & \mcP_2 & \mcP_3 & \mcP_4 & \mcP_5 & \mcP_6 & \mcP_7 & \mcP_8 & \cdots\\
\\
\mcH_0 & & \ar[ll]_{\tilde\Delta}  \ar[rr]^{\tilde R^2} R^2\mcH_0 & & R^4\mcH_0 & & 0 & &  \\
& \mcH_1 & & R^2\mcH_1 & & 0 & &  & \\
& & \mcH_2 & & 0 & &  & &  \\
& & & \mcH_3 & & R^2\mcH_3 & & R^4\mcH_3 & & \cdots\\
& & & & {\tilde{\mcH}_4} & & \ar[ll]_{\tilde\Delta}  \ar[rr]^{\tilde R^2} {R^2\tilde{\mcH}_4} & & {R^4\tilde{\mcH}_4} \\
& & & & & {\tilde{\mcH}_5} & & {R^2\tilde{\mcH}_5} & & \cdots\\
& & & & & & {\tilde{\mcH}_6} & & {R^2\tilde{\mcH}_6}\\
& & & & & & & \mcH_7 & & \cdots\\
& & & & & & & & \mcH_8 \\
& & & & & & & & & \ddots
}$$
\centerline{Here $\tilde\Delta=\Delta R^2\Delta$ and $\tilde R^2=R^2\Delta R^2$.}
\end{figure}

The following result deals with a~structure of $\osp(m|2n)$-module $\tilde\mcH_k$.

\begin{thm}\label{tildeH} Let the set $I_M$ of exceptional indices be defined as in \eqref{IM}. 
Denote $\mcH^0_k=R^2\Ker_{k-2} (\Delta R^2)$.

\smallskip\noindent
(i) Let $k\not\in I_M$. Then $\tilde\mcH_k=\mcH_k\simeq L^{m|2n}_{(k,0\ldots,0)}$ and $\mcH^0_k=0$. 

\smallskip\noindent
(ii) Let $k\in I_M$. Then all inclusions in $\mcH^0_k\subset\mcH_k\subset\tilde\mcH_k$ are proper and 
$$\mcH^0_k=R^{2k+M-2}\mcH_{2-M-k}.$$
Moreover, under the action of $\osp(m|2n)$, the module
$\tilde\mcH_k$ is indecomposable and its composition series is
$$\mcH^0_k\simeq L^{m|2n}_{(2-M-k,0\ldots,0)},\ \ \mcH_k/\mcH^0_k\simeq L^{m|2n}_{(k,0\ldots,0)}, \ \ \tilde\mcH_k/\mcH_k\simeq L^{m|2n}_{(2-M-k,0\ldots,0)}.$$ 

\end{thm}

\begin{remark}
It is easy to see that $\mcP_{k}$ is an irreducible module under the action of $\gl(m|2n)$ given by  the operators $X_i\pa_{X_j}$ for $i,j=1,\ldots, m+2n$.
Then the decomposition \eqref{eq_Fischer+} can be viewed as branching of $\mcP_k$ from $\gl(m|2n)$ to $\osp(m|2n)$. According to Theorem \ref{tildeH}, \eqref{eq_Fischer+} is not in general an irreducible but just indecomposable decomposition of $\mcP_{k}$ under $\osp(m|2n)$.
\end{remark}

\section{Proofs of Theorems \ref{Fischer} and \ref{tildeH}}

To prove our main results we need some lemmas.
First we show the following algebraic lemma that we use later for $U = \mcP_{k-2}$, $V=\mcP_{k}$, $L = \Delta$ and $K= R^2$.

\begin{lemma}\label{LKLK}
Let $U,V$ be two vector spaces. Let $L:V \rightarrow U$ and  $K:U \rightarrow V$ be homomorphisms such that $U = \Ker LK \oplus \Img LK$. Then the following statements hold true.

\smallskip\noindent
(i) We have that $V = \Ker LKL \oplus \Img KLK$.

\smallskip\noindent
(ii) If, in addition, $\Ker LK=0$, then $\Ker LKL=\Ker L$, $\Img KLK=\Img K$ and $V = \Ker L \oplus \Img K$.
\end{lemma}

\begin{proof}
(i) If $v \in V$, then $Lv \in U$ can be decomposed as $k + LK j$, where $k \in \Ker LK$ and $j \in U$. Decomposing $j$ further as $k'+LKu$, $k'\in \Ker LK$, we get $Lv=k+(LK)^2u$. This means $$LKL(v - KLK u) = LKk=0$$ and hence we get the decomposition
$$v = (v - KLKu) + KLKu$$
with $v - KLKu \in \Ker LKL$ and $KLKu \in \Img KLK$. 

If $f \in \Ker LKL \cap \Img KLK$, then there is a $g$ such that $f = KLKg$ and $(LK)^3g=0$. Then $(LK)^2g=0$ since it is both in kernel and image of $LK$. Repeating once more we get $LKg=0$, hence $f = 0$.

\smallskip\noindent
(ii) This follows easily from (i) because $LK$ is an isomorphism of $U$.
\end{proof}

\begin{lemma}\label{commutator} For $j,k\in\mN_0$, we have
$$[\Delta, R^{2j+2}]=  C(j,k) R^{2j}\text{\ \ on\ \ }\mcP_k$$ where $C(j,k)=(2j+2)(2 k + M + 2j)$.

\smallskip\noindent
In particular, it holds that $\Delta R^2=C(j,k)$ on $R^{2j} \mcH_k$ and $C(j,k)=0$ if and only if $k+j=-M/2$.
\end{lemma}

\begin{proof}
This follows from the commutation relation $[\Delta, R^2]= 4 \mE + 2 M$
by induction on $j$.
\end{proof}

\begin{lemma}\label{Fischer0} 
Let $M\in\mZ$ and $k\in\mN_0$. If $M\in-2\mN_0$, then assume that $k<2-M/2$.  Then we have that $\tilde\mcH_k=\mcH_k$ and
$$\mcP_k = \mcH_k \oplus R^2 \mcP_{k-2}.$$
\end{lemma}

\begin{proof}
This is known, cf.\ Theorem A. But, for the sake of completeness, we give a~proof by induction on $k$.
For $k=0,1$, we have $\mcP_k = \mcH_k=\tilde\mcH_k$. Now assume that the statement is true for all $0\leq j< k$.
Then $$\mcP_{k-2}=\bigoplus_{j=0}^{\lfloor(k-2)/2\rfloor} R^{2j}\mcH_{k-2j-2}$$
and, by Lemma \ref{commutator}, the operator $\Delta R^2$ is an isomorphism of the space $\mcP_{k-2}$.
Indeed, all the constants $C(j,k-2j-2)$ are non-zero.
Finally, applying Lemma \ref{LKLK}, we complete the proof.
\end{proof}

\begin{lemma}\label{isomorphism} Let $k,j\in\mN_0$. Then the following statements hold true.

\smallskip\noindent
(i) If $\tilde\mcH_k\not=\mcH_k$, then 
$\Delta R^2$ is an isomorphism of the space $R^{2j}\tilde\mcH_k$ 
(that is, $\Delta R^2$ is an isomorphism of the space $R^{2j}\tilde\mcH_k$ onto itself).

\smallskip\noindent
(ii) In general, $\Delta R^2$ is either an isomorphism of the space $R^{2j}\tilde\mcH_k$ or $\Delta R^2=0$ on $R^{2j}\tilde\mcH_k$.
Moreover, for a~non-zero $\tilde\mcH_k$, the latter possibility occurs if and only if $\tilde\mcH_k=\mcH_k$ and $k+j=-M/2$.
\end{lemma}

\begin{proof}
(i) Let $\tilde\mcH_k\not=\mcH_k$. By Lemma \ref{Fischer0}, we have that $M\in-2\mN_0$ and $k\geq 2-M/2$.
For a~given $H_k\in\tilde\mcH_k$, by Lemma \ref{commutator}, we get
\begin{equation}\label{eq_iso}
\Delta R^2 (R^{2j} H_k) = C(j,k) R^{2j} H_k + R^{2j} R^2 \Delta H_k
\end{equation}
where the constant $C(j,k)$ is non-zero. Since $R^2 \Delta H_k\in\mcH_k$ we have that $\Delta R^2 (R^{2j} H_k)\in R^{2j}\tilde\mcH_k$,
which gives $\Delta R^2 (R^{2j}\tilde\mcH_k)\subset R^{2j}\tilde\mcH_k$. The space $R^{2j}\tilde\mcH_k$ has a~finite dimension and so it is sufficient to show that the operator $\Delta R^2$ is injective on this space. To do this, let us notice that,  
if $H_k \in \mcH_k$, the right hand side of \eqref{eq_iso} is a nonzero multiple of $H_k$. If $H_k \in \tilde{\mcH}_k \setminus \mcH_k$, then the second term is in $R^{2j} \mcH_k$ and cannot cancel with the first term from $R^{2j}(\tilde\mcH_k\setminus\mcH_k)$, which finishes the proof of (i). 

\smallskip\noindent
(ii) This follows from (i) and Lemma \ref{commutator}.
\end{proof}

Now we are ready to prove Theorem \ref{Fischer}.

\begin{proof}[Proof of Theorem \ref{Fischer}]
We give a~proof by induction on $k$.
For $k=0,1$, we have $\mcP_k = \mcH_k=\tilde\mcH_k$. Now assume that the statement is true for all $0\leq j< k$.
Then 
\begin{equation}\label{eq_fischer}
\mcP_{k-2}=\bigoplus_{j=0}^{\lfloor(k-2)/2\rfloor} (R^2\Delta R^2)^j\tilde\mcH_{k-2j-2}.
\end{equation}
By Lemma \ref{isomorphism} (ii), we have that each direct summand $(R^2\Delta R^2)^j\tilde\mcH_{k-2j-2}$ is equal to either $R^{2j}\tilde\mcH_{k-2j-2}$ or $0$ and, consequently, it belongs to either $\Ker_{k-2}(\Delta R^2)$ or $\Img_{k-2}(\Delta R^2)$.
Then we have that $$\mcP_{k-2}=\Ker_{k-2}(\Delta R^2)\oplus\Img_{k-2}(\Delta R^2)$$
and we complete the proof of Theorem \ref{Fischer} by applying Lemma \ref{LKLK}. 
\end{proof}

\begin{proof}[Proof of Corollary \ref{cor_Fischer}]
We can prove this again by induction on $k$.
For $k=0,1$, we have $\mcP_k = \mcH_k$. Now assume that the statement is true for all $0\leq j< k$.
Then we have that
\begin{equation}\label{eq_Fischer++}
\mcP_{k-2}=\bigoplus_{\ell\in \tilde J_{k-2}} R^{k-2-\ell}\tilde\mcH_{\ell} \oplus
\bigoplus_{\ell\in J_{k-2}} R^{k-2-\ell}\mcH_{\ell}. 
\end{equation}
By Theorem \ref{Fischer}, we have
$\mcP_k = \tilde{\mcH}_k \oplus R^2 \Delta R^2 \mcP_{k-2}$. There are two possibilities.

\smallskip\noindent
(i) Let $k\not\in I_M$. Then $\tilde{\mcH}_k={\mcH}_k$ and, by Lemma \ref{isomorphism}, $\Delta R^2$ is an isomorphism of any direct summand of the decomposition \eqref{eq_Fischer++}. So $\Delta R^2 \mcP_{k-2}=\mcP_{k-2}$ and 
we get \eqref{eq_Fischer+} with $\tilde J_k=\tilde J_{k-2}$ and $J_k=J_{k-2}\cup \{k\}$.  In addition, we have ${\mcH}^0_k=0$.

\smallskip\noindent
(ii) Let $k\in I_M$. Then, by Lemma 4, we get from \eqref{eq_Fischer++} easily that $$\Ker_{k-2}(\Delta R^2)=R^{2k+M-4}\mcH_{2-M-k}.$$ 
Hence we obtain \eqref{eq_Fischer+} with $\tilde J_k=\tilde J_{k-2}\cup \{k\}$ and $J_k=J_{k-2}\setminus \{2-M-k\}$. 
In addition, we have that $\mcH^0_k=R^{2k+M-2}\mcH_{2-M-k}$.

\smallskip\noindent
The proof is now complete. Moreover, we have shown the statement of Lemma \ref{H0} (i) below as well.
\end{proof}

\begin{lemma}\label{H0}
(i) We have that $\mcH^0_k\not=0$ only if $k\in I_M$. Moreover, for $k\in I_M$, we get
$$\mcH^0_k=R^{2k+M-2}\mcH_{2-M-k}.$$

\smallskip\noindent
(ii) We have that $R^2\Delta (\tilde\mcH_k)=\mcH^0_k$.
In particular, $\tilde\mcH_k\not=\mcH_k$ if and only if $\mcH^0_k\not=0$.
\end{lemma}

\begin{proof} (i) This is known, see Theorem B. See also the proof of Corollary \ref{cor_Fischer} above.

\smallskip\noindent
(ii) Obviously, $R^2\Delta (\tilde\mcH_k)\subset\mcH^0_k$ because $\Delta (\tilde\mcH_k)\subset \Ker_{k-2}(\Delta R^2)$.
Assume now that $R^2 H_{k-2}\in\mcH^0_k$. Since $\Delta:\mcP_k\to\mcP_{k-2}$ is onto (see \cite[Theorem 2]{DS3}) there is $H_k\in\mcP_k$ such that $\Delta H_k=H_{k-2}$.
Then $H_k\in\tilde\mcH_k$ and $R^2\Delta H_k=R^2H_{k-2}$.
\end{proof}

Now we are going to prove Theorem \ref{tildeH}.

\begin{proof}[Proof of Theorem \ref{tildeH}]
Obviously, we have $\tilde\mcH_k/\mcH_k\simeq \mcH^0_k$ because the invariant operator $R^2\Delta:\tilde\mcH_k\to\mcH^0_k $ is onto and has the kernel $\mcH_k$ (see Lemma \ref{H0}). 
Then, according to Theorem B, it remains to prove that $\tilde\mcH_k$ is an $\osp(m|2n)$-indecomposable module.

Assume that $M\in-2\mN_0$, $m\not=0$ (that is, $m\geq 2$) and $k\in I_M$. We show that $\tilde\mcH_k$ is an $\osp(m|2n)$-indecomposable module.
To do this, let $\tilde\mcH_k=U\oplus V$ for some $\osp(m|2n)$-modules $U,V$. Then we prove that one of the modules $U,V$ is trivial.

Indeed, under the action of $\mfg_0=\so(m)\times\spp(2n)$, the even part of $\mfg=\osp(m|2n)$,
the finite-dimensional module $\tilde\mcH_k$ is completely reducible and $\tilde\mcH_k=\mcH_k\oplus W$ for some $\mfg_0$-module $W$. Since $R^2\Delta:W\to\mcH^0_k $ is a~$\mfg_0$-invariant isomorphism we have that $W\simeq\mcH^0_k\simeq\mcH_{2-M-k} $ 
and, consequently, $\tilde\mcH_k\simeq\mcH_k\oplus\mcH_{2-M-k} $ as $\mfg_0$-modules.
By Proposition~\ref{decomp0}, we know explicitly $\mfg_0$-irreducible decomposition of  $\mcH_k$ and $\mcH_{2-M-k}$.
In particular, $\mcH^b_{k}$ can be viewed as a~subset of $\tilde\mcH_k$ such that it is the unique $\mfg_0$-irreducible submodule of $\tilde\mcH_k$ which is equivalent to $\mcH^b_{k}\otimes \mcH^f_0$.
Therefore, $\mcH^b_{k}$ is contained in one of the modules $U$ or $V$, let us say, $U$. By \cite[Proof of Theorem 5.1]{Cou}, we know that $\mcH_{k}=\mcU(\mfg)(\mcH^b_{k})$ and hence $\mcH_{k}\subset U$. Here $\mcU(\mfg)$ is the universal enveloping superalgebra of $\mfg$.

Now we are ready to prove that $V=0$. Indeed, since $V\cap\mcH_k=0$ the map $R^2\Delta:V\to\mcH^0_k$ is injective. On the other hand,
by \eqref{Casimir},
$R^2\Delta$ on $\mcP_k$ is equal to the Casimir operator of degree 2 for $\mfg$ up to a~constant, which gives $R^2\Delta(V)\subset V\cap\mcH^0_k=0$ and, finally, $V=0$. This completes the proof.
\end{proof}

\subsection*{Acknowledgments}

The authors are very grateful for useful advise and suggestions from Vladim\'ir Sou\v cek.

\end{document}